\documentclass[11pt,reqno]{amsart}
\usepackage[all]{xy}
\usepackage{amssymb}
\usepackage{amsthm}
\usepackage[normalem]{ulem}
\usepackage{amsmath,mathtools}
\usepackage{amscd,enumitem}
\usepackage{verbatim}
\usepackage{eurosym}
\usepackage{float}
\usepackage{color}
\usepackage{dcolumn}
\usepackage[mathscr]{eucal}
\usepackage[all]{xy}
\usepackage{bbm}
\usepackage[textheight=8.5in, textwidth=6.7in]{geometry}
\newtheorem*{conj*}{Conjecture}
\newtheorem{theorem}{Theorem}[section]

\theoremstyle{definition}
\newtheorem*{remark}{Remark}
\theoremstyle{plain}

\newtheorem{lemma}[theorem]{Lemma}

\newtheorem{prop}[theorem]{Proposition}

\newcommand{\Z}{\mathbb{Z}}

\newcommand{\R}{\mathbb{R}}
\newcommand{\N}{\mathbb{N}}

\newcommand{\C}{\mathbb{C}}
\newcommand{\im}[1]{\text{Im}\(#1\)}
\newcommand{\re}[1]{\text{Re}\(#1\)}

\DeclareMathOperator\Log{Log}
\numberwithin{equation}{section}

\newtheoremstyle{example}
  {\topsep}   
  {\topsep}   
  {\normalfont}  
  {0pt}       
  {\bfseries} 
  {.}         
  {5pt plus 1pt minus 1pt} 
  {}          
\theoremstyle{example}

\def\({\left(}
\def\){\right)}

\def\calC{\mathcal{C}}

\usepackage{centernot}

\newcommand\blfootnote[1]{%
	\begingroup
	\renewcommand\thefootnote{}\footnote{#1}%
	\addtocounter{footnote}{-1}%
	\endgroup
}

\author{Jashan Bal$^\dagger$}
\author{Fern Haraldson$^\dagger$}
\author{Joshua Males}
\author{Ian Thompson}

\address{Department of Mathematics, University of Manitoba, Winnipeg, Manitoba, Canada R3T 2N2}

\email{balj@myumanitoba.ca}
\email{haraldsf@myumanitoba.ca}
\email{joshua.males@umanitoba.ca}
\email{thompsoi@myumanitoba.ca}
\thanks{J.M. is supported by the Pacific Institute for the Mathematical Sciences (PIMS). The research and findings may not reflect those of the Institute. I.T. was partially supported by an NSERC CGS-D Scholarship and the Manitoba eXperimental Mathematics Laboratory.}

\begin{document}
\title[Jensen polynomials associated with Wright's circle method]{Jensen polynomials associated with Wright's circle method: Hyperbolicity and Tur\'an inequalities}
\blfootnote{$^\dagger$ Undergraduate author at the University of Manitoba.}

\subjclass[]{}

\begin{abstract}
We study the Fourier coefficients of functions satisfying a certain version of Wright's circle method with finitely many major arcs. We show that the Jensen polynomials associated with such Fourier coefficients are asymptotically hyperbolic, building on the framework of Griffin--Ono--Rolen--Zagier and others. Consequently, we prove that the Fourier coefficients asymptotically satisfy all higher-order Tur\'an inequalities. As an application, we apply our results to both $(q^t;q^t)_\infty^{-r}$, which counts $r$-coloured partitions into parts divisible by $t$, and to the function $(q^{a};q^{p})_\infty^{-1}$ where $p$ is prime and $0\leq a<p$, a ubiquitous function throughout number theory.
\end{abstract}
\maketitle

\vspace{-3.5mm}
\section{Introduction}\label{S:Intro}

A classical result of P\'olya \cite{polya1927algebraisch} reformulated the Riemann Hypothesis in terms of hyperbolicity of the so-called Jensen polynomials associated with a particular Taylor expansion. Accordingly, a polynomial is said to be \emph{hyperbolic} if every zero is a real number and, given a sequence of real numbers $\{\alpha(n)\}$, the \emph{Jensen polynomial of degree $d$ and shift $n$} is defined by\[J_\alpha^{d,n}(X) = \sum_{j=0}^d {{d}\choose{j}}\alpha(n+j)X^j\]where $d,n\in\N$. Then, P\'olya showed that the Riemann Hypothesis is equivalent to showing hyperbolicity of every Jensen polynomial associated with the sequence of Taylor coefficients arising from the expansion 
\begin{align}\label{EQ:RH}
	(-1+4z^2)\pi^{-\frac{(1+2z)}{4}}\Gamma\(\frac{1+2z}{4}\)\zeta\(\frac{1}{2}+z\) = \sum_{n=0}^\infty \frac{\gamma(n)}{n!}z^{2n}.
\end{align} 
It was recently shown that, for any fixed $d\in\N$, the Jensen polynomials $J_\gamma^{d,n}(X)$ are hyperbolic for all but finitely many $n$ \cite{griffin2019jensen}. Moreover, up to normalization, the Jensen polynomial $J_\gamma^{d,n}(X)$ tends towards the Hermite polynomial $H_d(X)$ as $n\rightarrow\infty$. Within \cite{griffin2019jensen}, it was also shown that a large family of modular objects induce asymptotically hyperbolic Jensen polynomials and that again, up to normalization, the Jensen polynomials tend to Hermite polynomials. Several expansions have since appeared in the growing literature surrounding the field (see e.g. \cite{griffin2022jensen},\cite{larson2019hyperbolicity},\cite{o2022limits},\cite{o2021zeros},\cite{wagner2020jensen}).

We expand on that list by proving that the Fourier coefficients of another large class of functions give rise to Jensen polynomials that are asymptotically hyperbolic. Classically, the philosophy behind the Circle Method was initiated by Hardy and Ramanujan \cite{hardy2000asymptotic} in their seminal work which gave a precise asymptotic formula for the integer partitions. A \emph{partition} $\lambda = (\lambda_1, \ldots, \lambda_s)$ of a positive integer $n$ is a non-decreasing tuple satisfying $\sum_{i=1}^s |\lambda_i|=n$. Letting $p(n)$ denote the number of partitions for a positive integer $n$, Hardy and Ramanujan \cite{hardy2000asymptotic} showed that \begin{align*}
	p(n)\sim \frac{1}{4n\sqrt{3}}e^{\pi\sqrt{\frac{2n}{3}}},
\end{align*}
as $n\rightarrow\infty$.

An expansion of this philosophy was introduced by Wright \cite{wright1971stacks} and has seen an increase in popularity in recent years. Wright's method obtains an asymptotic description for the Fourier coefficients of functions that possess suitable growth conditions around cusps. Roughly, one uses Cauchy's Theorem to interpret the Fourier coefficients as the integral of the generating function over a circle of radius less than one, and then divides the path of integration into two arcs, commonly referred to as the major and minor arcs. The major arc consists of neighbourhoods around singularities (or alternatively, roots of unity) where the generating function has relatively large growth, and the minor arc is where the generating function exhibits non-dominant asymptotic behaviour. Wright's work does not yield an exact asymptotic formula and therefore introduces weaker bounds compared to the full Circle Method, but it is often easier to work with. Wright's  variant of the Circle Method is a very flexible tool, and may be applied to a wide variety of both modular and non-modular generating functions that carry suitable analytic restraints. For our version, we focus on those functions that satisfy certain technical growth conditions given explicitly in Proposition \ref{P:Wright}. This builds on work of Bringmann, Craig, Ono, and the third author \cite{bringmann2022distributions} and results of Ngo and Rhoades \cite{Roades}.

\begin{theorem}\label{T:Hyperbolic}
Let $F(q) = \sum_{n=0}^\infty c(n)q^n$ be a function satisfying the conditions of Proposition \ref{P:Wright}. Define $C(n) = c(Kn)$ with $K \in \N$ as in Proposition \ref{P:Wright}. Then for each fixed $d\geq 1$, we have that the Jensen polynomial $J_C^{d,n}(X)$ is hyperbolic for all but finitely many $n$. Moreover, up to normalization, the Jensen polynomials $J_C^{d,n}(X)$ converge uniformly to the Hermite polynomial $H_d(X)$ uniformly over compact subsets of $\R$ as $n\rightarrow\infty$.
\end{theorem}

We connect Theorem \ref{T:Hyperbolic} with higher order Tur\'an inequalities, which have been a topic of great interest as of late \cite{chen2019higher},\cite{craig2021note},\cite{o2022limits},\cite{ono2022turan}. One of the primary reasons is that the higher order Tur\'an inequalities are intimately connected with real entire functions of Laguerre--P\'olya class. Indeed, an alternative characterization of the Riemann Hypothesis is that the Riemann Xi function lies in the Laguerre--P\'olya class \cite{dimitrov1998higher},\cite{szego1948on} and that, if the Riemann Hypothesis were true, we must have that the Maclaurin coefficients satisfy all the higher order Tur\'an inequalities \cite{dimitrov1998higher},\cite{polya1914uber}.

We now describe how to obtain the higher order Tur\'an inequalities. First, a sequence $\{ a_n\}_{n\geq 0}$ is said to satisfy the second order Tur\'an inequality if \[a_n^2\geq a_{n-1}a_{n+1}\]for every $n\geq 1$. The second order Tur\'an inequality is also commonly referred to as log-concavity and has been well studied in many settings \cite{bringmann2019peak},\cite{cesana2021asymptotic},\cite{craig2021note},\cite{dawsey2019effective},\cite{desalvo2015log}. For higher orders, a classical theorem of Hermite states that a polynomial $J(X)$ with real coefficients is hyperbolic precisely when the associated Hankel matrix is positive-definite. The minors of the Hankel matrix provide a corresponding sequence of inequalities associated with the coefficients of the polynomial $J(X)$. By considering Jensen polynomials, the resulting inequalities are referred to as the higher order Tur\'an inequalities. Specifically, if $\beta_1, \ldots, \beta_n$ are the roots of $J(X)$, then let $S_0=n$ and\[S_k = \beta_1^k + \ldots + \beta_n^k\]for every $k\geq1$. Then, $J(X)$ is hyperbolic precisely when the Hankel matrix \[M_n(J) = \begin{bmatrix}S_0 & S_1 & S_2 & \ldots & S_{n-1}\\
S_1 & S_2 & S_3 & \ldots & S_n\\
S_2 & S_3 & S_4 & \ldots & S_{n+1}\\
\vdots & \vdots & \vdots & \vdots & \vdots\\
S_{n-1} & S_n & S_{n+1} & \ldots & S_{2n-2}
\end{bmatrix}\]is positive-definite for each $n\in\N$. The latter condition occurs precisely when every minor of each of the matrices $M_n(J)$ is positive, and so we obtain a collection of inequalities \[\Delta_k^{(n)} = \left|\begin{array}{ccc}
S_0 & \ldots & S_{k-1}\\
\vdots & \vdots & \vdots\\
S_{k-1} & \ldots & S_{2k-2}
\end{array}\right|\geq 0, \ \ \ \ \ \ \ \ \ \ \ \ \ \ \ 1\leq k \leq n.\]The resulting inequalities are the higher order Tur\'an inequalities and are completely determined by the coefficients of the polynomial $J(X)$. For more details, we refer the reader to \cite{craig2021note}. Our main contribution regarding higher order Tur\'an inequalities is then seen to be an immediate consequence to Theorem \ref{T:Hyperbolic} and results of Griffin, Ono, Rolen, and Zagier \cite{griffin2019jensen}. While unsurprising to experts, it appears to not be explicitly written in the literature, and so we record the result explicitly here.

\begin{theorem}\label{T:Turan}
Let $F(q) = \sum_{n=0}^\infty c(n)q^n$ be a function satisfying the conditions of Proposition \ref{P:Wright}. Let $C(n) = c(Kn)$ with $K \in \N$ as in Proposition \ref{P:Wright}. Then, the sequence of coefficients $C(n)$ asymptotically satisfy all of the higher order Tur\'an inequalities.
\end{theorem}

We end by applying our results to two classes of infinite products, each of which exhibit very different modular behaviour. To this end, we let $(a; q)_\infty :=\prod_{n\geq 0}(1-aq^n)$ denote the usual $q$-Pochhammer symbol. First, we consider 
\begin{align*}
	H_{r,t} (q) \coloneqq \frac{1}{\left(q^t;q^t\right)_\infty^r} = \sum_{n \geq 0} c_{r,t}(n) q^n,
\end{align*}
where $c_{r,t}(n)$ counts the number of $r$-coloured partitions into parts which are divisible by $t$. We remark that $H_{r,t}(q)$ is (up to a simple $q$-power) a modular form, and so falls under the framework provided by \cite{griffin2019jensen}. Moreover, the asymptotics of the coefficients $c_{r,t}(n)$ are well-known. Here, we make the result explicit as an example of the general theorems given in this paper.

\begin{theorem}\label{T:AppH}
Let $H_{r,t}(q):= \frac{1}{(q^t; q^t)_\infty^r} =  \sum_{n=0}^\infty c_{r,t}(n)q^n$ where $r$ and $t$ are positive integers. Define a sequence of integers by $C_{r,t}(n) = c_{r,t}(tn)$. Then, the following statements hold.
\begin{enumerate}
\item For each fixed $d\geq 1$, the Jensen polynomial $J_{C_{r,t}}^{d,n}(X)$ is hyperbolic for all but finitely many $n$. Moreover, up to normalization, the Jensen polynomials $J_{C_{r,t}}^{d,n}(X)$ converge uniformly to the Hermite polynomial $H_d(X)$ uniformly over compact subsets of $\R$ as $n\rightarrow\infty$.
\item The sequence of coefficients $C_{r,t}(n)$ asymptotically satisfy all higher order Tur\'an inequalities.
\end{enumerate}
\end{theorem}

Secondly, we consider the ubiquitous infinite product
\begin{align}
	G_{a,p}(q):= \frac{1}{(q^{a}; q^{p})_\infty} = \sum_{n=0}^\infty b_{a,p}(n)q^n,
\end{align}
where $p$ is prime and $0\leq a<p$. In contrast to $H_{r,t}(q)$, the function $G_{a,p}(q)$ is non-modular, and so we cannot appeal to modular transformation formulae to obtain bounds on major and minor arcs. We show that the function $G_{a,p}(q)$ satisfies the conditions of Wright's Circle Method (Proposition \ref{P:Wright}). Theorems \ref{T:Hyperbolic} and \ref{T:Turan} will then imply the following result. 

\begin{theorem}\label{T:App}
Let $G_{a,p}(q):= \frac{1}{(q^{a}; q^{p})_\infty} =  \sum_{n=0}^\infty b_{a,p}(n)q^n$ where $p$ is prime and $0\leq a<p$. Then, the following statements hold.
\begin{enumerate}
\item The coefficients $b_{a,p}(n)$ have the asymptotic behaviour
\begin{align*}
	b_{a,p}(n) \sim\frac{\Gamma\left(\frac{a}{p}\right) \left(\frac{a}{p}\right)^{\frac{1}{2}-\frac{a}{p}}p^{\frac{1}{2}-\frac{a}{p}}}{\sqrt{2\pi}}  \left(\frac{\pi^2}{6pn}\right)^{\frac{1}{2}\left(\frac{1}{2} - \frac{a}{p}\right)+\frac{1}{2}} I_{-\left(\frac{1}{2} - \frac{a}{p} +1\right)} \left( \frac{\pi}{p} \sqrt{\frac{2n}{3}}  \right) + O\left( e^{ \frac{\pi}{p} \sqrt{\frac{2n}{3}}} n^{-\left(\frac{1}{2}(\frac{1}{2}-\frac{a}{p})+1\right)} \right)
\end{align*}
as $n \to \infty$.
\item For each fixed $d\geq 1$, the Jensen polynomial $J_{b_{a,p}}^{d,n}(X)$ is hyperbolic for all but finitely many $n$. Moreover, up to normalization, the Jensen polynomials $J_{b_{a,p}}^{d,n}(X)$ converge uniformly to the Hermite polynomial $H_d(X)$ uniformly over compact subsets of $\R$ as $n\rightarrow\infty$.
\item The sequence of coefficients $b_{a,p}(n)$ asymptotically satisfy all higher order Tur\'an inequalities.
\end{enumerate}
\end{theorem}

\begin{remark}
	The function $G_{a,p}(q)$ is a subclass of functions that were recently studied in a beautiful paper of Chern \cite{chern2019nonmodular}. In \cite[Theorem 1.1]{chern2019nonmodular}, Chern obtained the main-term asymptotic toward any root of unity using different methods and many technical bounds. In the present paper, we apply our techniques and obtain a more precise asymptotic formula for the coefficients of $G_{a,p}(q)$, which then allows us to conclude asymptotic hyperbolicity of the associated Jensen polynomial.
\end{remark}

We provide a brief outline of our work. In Section \ref{S:Prelim}, we outline a few of our primary tools. In Section \ref{S:Proof}, we prove Theorem \ref{T:Hyperbolic} and provide an alternative proof for our version of Wright's circle method. In Section \ref{S:App}, we apply our work to the functions $H_{r,t}(q)$ and $G_{a,p}(q)$ by proving Theorems \ref{T:AppH} and \ref{T:App}.

$\left.\right.$\\
{\bf Acknowledgments.} We are grateful for the support of the Manitoba eXperimental Mathematics Laboratory.
\vspace{-2mm}

\section{Preliminaries}\label{S:Prelim}

\subsection{Wright's Circle Method}\label{SS:Wright}

The Circle Method of Hardy-Ramanujan constitutes a powerful method to obtain asymptotic formulae (and even exact formulae) for generating functions that carry suitable growth conditions at all cusps. Wright \cite{wright1971stacks} provided a variant that has proven to be popular within the modern canon. We provide a minor modification to that presented in \cite[Proposition 4.4]{bringmann2022distributions} to account for analytic functions enjoying a certain symmetry. Throughout, we will let $I_v(z)$ denote the $I$-Bessel function and $\zeta_K \coloneqq e^{\frac{2 \pi i}{K}}$ where $K \in \N$.
\begin{prop}\label{P:Wright}
Suppose that $F(q)$ is analytic for $q = e^{-z}$ where $z = x+iy\in\C$ satisfies $x>0$ and $|y|<\pi$, and suppose that $F(q)$ has an expansion $F(q) = \sum_{n=0}^\infty c(n)q^n$ near $1$. Let $K, M, N>0$ be fixed constants. Suppose that $F(\zeta_Kq) = F(q)$ and consider the following hypotheses:\begin{enumerate}
\item As $z\rightarrow 0$ in the bounded cone $ 0 \leq y \leq Mx$ (major arc), we have \[ F(e^{-z}) = z^B e^{\frac{A}{z}}\left( \sum_{j=0}^{N-1}\alpha_j z^j + O_M(|z|^N) \right),\]where $\alpha_j\in\C, A\in\R^+$, and $B\in\R^+.$
\item As $z\rightarrow 0$ in the bounded cone $Mx\leq y\leq \frac{2\pi}{K} - Mx$ (minor arc), we have \[ |F(e^{-z})| \ll_M e^{\frac{1}{\text{Re}(z)}(A-\kappa)},\]for some $\kappa\in\R^+.$
\end{enumerate}
Define a sequence $C(n) \coloneqq c(Kn)$. Then as $n\rightarrow\infty$ we have that,
\begin{align*}
	C(n) =\sum_{j=0}^{N-1}\(\alpha_j \(\frac{A}{Kn}\)^{\frac{1}{2}(j+B+1)}I_{-(j+B+1)}(2\sqrt{AKn})\) + O\( e^{2\sqrt{AKn}}n^{-\frac{1}{2}(N+B+1)} \).
\end{align*}
\end{prop}

The form of Proposition \ref{P:Wright} stated in \cite[Proposition 4.4]{bringmann2022distributions} is the case for when there is a single major arc near $q=1$. Accordingly, in Section \ref{S:Proof}, we record the additional details and provide an alternative proof to that presented in \cite{bringmann2022distributions}. Moreover, we show that the sequence of Jensen polynomials corresponding to the coefficients $C(n)$ are hyperbolic for all but finitely many $n$. Finally, in Section \ref{S:App}, we verify conditions (1) and (2) above for a concrete class of functions. 

We record one additional fact regarding the $I$-Bessel function that will be needed in the proof of Theorem \ref{T:Hyperbolic}. That is, one can obtain an expansion in terms of exponentials by using the well-known asymptotic of the $I$-Bessel function \cite[10.40.1]{NIST:DLMF} as follows
\begin{align}\label{Eqn: I bessel asymp}
	 I_v(z)\sim \frac{e^z}{(2\pi z)^{\frac{1}{2}}} \sum_{k=0}^\infty (-1)^k \frac{a_k(v)}{z^k}, \qquad  \lvert \arg{z} \vert \leq \frac{\pi}{2}-\delta \text{ where } \delta>0,
\end{align}
where
\begin{align*}
	a_k(v) \coloneqq \frac{\left(\frac{1}{2} -v\right)_k \left(\frac{1}{2}+v\right)_k}{  (-2)^k k!},
\end{align*}
with $(v)_k \coloneqq v(v+1)\cdots(v+k)$ denoting the usual Pochhammer symbol.

\subsection{Hyperbolicity of Jensen Polynomials}\label{SS:GORZ}

Recently, it was unveiled that a large class of Jensen polynomials become Hermite polynomials asymptotically \cite{griffin2019jensen}. The Hermite polynomials $H_d(X)$, which are defined by the generating function $e^{-t^2+Xt}=\sum_{d=0}^\infty H_d(X) \frac{t^d}{d!}$, are hyperbolic and so this provides us a method for obtaining asymptotic hyperbolicity of certain Jensen polynomials. This approach was successfully implemented in \cite{griffin2019jensen} to provide asymptotic hyperbolicity of Jensen polynomials associated with the Riemann zeta function. Their main tool is the following \cite[Theorem 3]{griffin2019jensen} (their result is misstated but it is correctly stated in \cite[Proof of Theorem 1]{griffin2019jensen}).

\begin{theorem}\label{T:GORZJensen}
Let $\{\alpha(n)\}, \{ A(n)\}$ and $\{\delta(n)\}$ be sequences of positive numbers with $\delta(n)$ converging to zero. Further, suppose that \[\log\(\frac{\alpha(n+k)}{\alpha(n)}\) = A(n)k - \delta(n)^2k^2 + \sum_{j=3}^d g_j(n)k^j + o(\delta(n)^d) \ \ \ \ \ \ \text{as } n\rightarrow\infty,\]where $d\in\N, 0\leq k\leq d$ and $g_i(n) = o\(\delta(n)^i\)$. Then, we have that \[ \frac{\delta(n)^{-d}}{\alpha(n)} J_\alpha^{d,n}\( \frac{\delta(n)X-1}{\text{exp}(A(n))}\) \rightarrow H_d(X) \ \ \ \ \ \text{as } n\rightarrow\infty,\]uniformly in $X$ for every compact subset of $\R$. Moreover, $J_\alpha^{d,n}(X)$ is hyperbolic for all but finitely many $n$.
\end{theorem}

We will require Theorem \ref{T:GORZJensen} in order to verify that the sequence of coefficients $C(n)$ found within Proposition \ref{P:Wright} yield asymptotically hyperbolic Jensen polynomials.

\subsection{The Euler-Maclaurin summation formula}\label{SS:EM}

To prove Theorem \ref{T:App}, we will require one version of the Euler-Maclaurin summation formula that was proven in \cite[Lemma 2.2]{bringmann2022distributions}. This allows us to obtain precise asymptotics for the classes of functions $H_{r,t}(q)$ and $G_{a,p}(q)$ close to certain roots of unity. For these purposes, we introduce additional notation.

If $s,z\in\C$ are such that $\re{s}>1$ and $\re{z}>0$, then the \emph{Hurwitz zeta function} is defined by $\zeta(s,z):= \sum_{n=0}^\infty \frac{1}{(n+z)^s}$, the \emph{digamma function} is defined by $\psi(x):=\frac{\Gamma'(x)}{\Gamma(x)}$, and the Euler-Mascheroni constant is denoted by $\gamma$. Moreover, the \emph{$n$-th Bernoulli polynomial} $B_n(x)$ is determined by the generating function $\frac{te^{xt}}{e^t-1} = \sum_{n=0}^\infty B_n(x)\frac{t^n}{n!}.$ We say that a function $f$ defined on a domain in $\C$ is of \emph{sufficient decay} provided that there is some $\varepsilon>0$ such that $f(w)\ll w^{-1-\varepsilon}$ as $|w|\rightarrow\infty$ in the domain. Additionally, define \[D_\theta := \{ z=re^{i\alpha} : r\geq 0, |\alpha|\leq \theta\}\]where $0\leq\theta<\frac{\pi}{2}.$ 

\begin{lemma}\label{L:EM}
Let $A\in\R^+$, $0<a\leq 1$, and $0\leq\theta<\frac{\pi}{2}.$ Assume that there is $n_0\in\Z$ such that $f(z)\sim\sum_{n=n_0}^\infty d(n)z^n$ as $z\rightarrow0$ in $D_\theta$. Furthermore, assume that $f$ and all of its derivatives are of sufficient decay in $D_\theta$. Then, \[ \sum_{n=0}^\infty f(z(n+a))\sim \sum_{n=n_0}^\infty d(n)\zeta(-n, a)z^n + \frac{I_{f,A}^*}{z} - \frac{d(-1)}{z}\(\Log(Az) + \psi(a) + \gamma\) - \sum_{n=0}^\infty d(n)\frac{B_{n+1}(a)}{n+1}z^n,\]as $z\rightarrow0$ uniformly in $D_\theta$, where \[ I_{f,A}^* := \int_0^\infty \( f(u) - \sum_{n=n_0}^{-2}d(n)u^n - \frac{d(-1)e^{-Au}}{u}\)du.\]
\end{lemma}

\section{Proof of Proposition \ref{P:Wright} and Theorem \ref{T:Hyperbolic}}\label{S:Proof}

\begin{proof}[Proof of Proposition \ref{P:Wright}]
We consider the circle $\calC$ centred at the origin, surrounding zero exactly once in the counter-clockwise orientation, and of radius $|q| = e^{-\lambda}$ where $\lambda = \sqrt{\frac{A}{Kn}}$. In particular, we let $q = e^{-z}$ with $z=\lambda(1+i\beta M)$. The major arc near $q=1$ will be that where $-1\leq \beta \leq 1$ (in analogy with the arcs used in \cite{bringmann2016dyson}).

Let $\calC_M$ and $\calC_m$ denote the major and minor arcs, respectively, and note that $\calC_M$ consists of neighbourhoods around the $K$-th roots of unity. Denote the major arc corresponding to the root of unity $\zeta_K^h$ by $\calC_h$ for $0\leq h\leq K-1$. As $F(q)$ has major arcs at the $K$-th roots of unity, the behaviour around each major arc is analogous by symmetry. Indeed, observe that\[ \frac{1}{2\pi i}\int_{\calC_h} \frac{F(q)}{q^{m+1}}dq = \frac{1}{2\pi i}\int_{\calC_0} \frac{F(\zeta_K^h q)}{(\zeta_K^h q)^{m+1}}d(\zeta_K^h q) = \frac{\zeta_K^{-hm}}{2\pi i}\int_{\calC_0} \frac{F(q)}{q^{m+1}}dq\]for each $0\leq h\leq K-1$ and every $m\in\N$. By Cauchy's Theorem, we then find that\[ C(n) = \frac{1}{2\pi i}\int_{\calC} \frac{F(q)}{q^{Kn+1}}dq=\frac{K}{2\pi i}\int_{\calC_0}\frac{F(q)}{q^{Kn+1}}dq+ \frac{1}{2\pi i} \int_{\calC_m}\frac{F(q)}{q^{Kn+1}}dq.\]Now, define\[A_j(n) = \frac{K}{2\pi i}\int_{\calC_0}\frac{z^{j+B}e^{\frac{A}{z}}}{q^{Kn+1}}dq\]and express\[ C(n) - \sum_{j=0}^{N-1}\alpha_j A_j(n) = \varphi_1(n)+\varphi_2(n)\]where \[ \varphi_1(n) = \frac{1}{2\pi i}\int_{\calC_m} \frac{F(q)}{q^{Kn+1}}dq, \ \ \ \ \varphi_2(n) = \frac{K}{2\pi i}\int_{\calC_0}\( F(q)z^{-B}e^{-\frac{A}{z}} - \sum_{j=0}^{N-1}\alpha_jz^j \) \frac{z^B e^{\frac{A}{z}}}{q^{Kn+1}} dq.\]By condition (1), we have that on $\calC_0$, \[ \left| F(q)z^{-B}e^{-\frac{A}{z}} - \sum_{j=0}^{N-1} \alpha_j z^j \right|  = O(|z|^N).\]Since $\re{z} = \lambda$ on $\calC$, \[ \left| \exp\(\frac{A}{z}+Knz \) \right|  \leq \exp\(\frac{A}{\lambda} + Kn\lambda\) = \exp(2\sqrt{AKn}).\]Since the length of $\calC_0$ is $\approx\lambda$ and $\im{z} = O(\lambda)$, we have that $|z|\sim (Kn)^{-\frac{1}{2}}$. Then, Cauchy's integral estimate yields that \[ \varphi_2(n) = O(n^{-\frac{1}{2}(N+B+1)}e^{2\sqrt{AKn}})\]as desired.

Next note that the length of $\calC_m$ is $\approx1$ and so by condition (2), \[ \varphi_1(n) = O\(e^{\frac{1}{\lambda}(A-\kappa)} |q|^{-Kn}\) = O\( e^{\(2-\frac{\kappa}{A}\)\sqrt{AKn}}\).\]Consequently, we have that
 \begin{equation}\label{EQ:Wright}
C(n) - \sum_{j=0}^{N-1} \alpha_j A_j(n) = O\( n^{-\frac{1}{2}(N+B+1)}e^{2\sqrt{AKn}}\).
\end{equation}
Now, notice that we may parametrize $q\in\calC_0$ counter-clockwise by $\lambda-i \lambda M\leq z\leq \lambda+i \lambda M$. Thus, we get that \[A_j(n) = \frac{K}{2\pi i}\int_{\lambda(1-iM)}^{\lambda(1+iM)} z^{j+B}e^{\frac{A}{z}+Knz}dz\]and, by making a change of variables $z\mapsto\lambda z$, \[ A_j(n) = \frac{K \lambda^{j+B+1}}{2\pi i}\int_{1-iM}^{1+iM}z^{j+B}e^{\sqrt{AKn}\(z+\frac{1}{z}\)}dz.\]
By \cite[Lemma 4.2]{bringmann2016dyson}\footnote{The papers \cite{bringmann2016dyson} and \cite{Roades} rely on obtaining an estimate for the integrals $A_j$ in terms of $I$-Bessel functions by arguments given in \cite{Roades}. While the result given in \cite{Roades} is correct, we believe there is a very slight misstep in their proof, and we use an alternative route to obtaining the $I$-Bessel functions.}, we get that 
\begin{align*}
A_j(n) & = \lambda^{j+B+1}\( I_{-(j+B+1)}(2\sqrt{AKn}) + O\( e^{\sqrt{AKn}\(1+ \frac{1}{1+M^2}\)}\)\)\\
& = \(\frac{A}{Kn}\)^{\frac{1}{2}(j+B+1)}I_{-(j+B+1)}(2\sqrt{AKn}) + O\( n^{-\frac{1}{2}(j+B+1)}e^{\sqrt{AKn}\( 1+\frac{1}{1+M^2}\)} \).
\end{align*}The desired estimate is then obtained by plugging into (\ref{EQ:Wright}).
\end{proof}

Next, we show that the Jensen polynomials associated with the asymptotic formula given for $C(n)$ will be asymptotically hyperbolic. To do so, we will apply Theorem \ref{T:GORZJensen} to the Fourier coefficients found in Proposition \ref{P:Wright}.

\begin{proof}[Proof of Theorem \ref{T:Hyperbolic}]
Take $N=1$ in Proposition \ref{P:Wright} to obtain that \[ C(n) =\alpha_0 \(\frac{A}{Kn}\)^{\frac{1}{2}(B+1)}I_{-(B+1)}(2\sqrt{AKn}) + O\( e^{2\sqrt{AKn}}n^{-\frac{1}{2}(B+2)} \).\]By appealing to the asymptotic form of the $I$-Bessel function given in \eqref{Eqn: I bessel asymp}, we find that 
\begin{align*}
C(n)& \sim\alpha_0\(\frac{A}{Kn}\)^{\frac{1}{2}(B+1)}(2\pi\sqrt{AKn})^{-\frac{1}{2}}e^{2\sqrt{AKn}} \sum_{m=0}^\infty a_m n^{-\frac{m}{2}} + O\( e^{2\sqrt{AKn}}n^{-\frac{1}{2}(B+2)}\)\\
& = \frac{\alpha_0}{\sqrt{2\pi}} A^{\frac{2B+1}{4}}(Kn)^{-\frac{2B+3}{4}} e^{2\sqrt{AKn}}\sum_{m=0}^\infty a_m n^{-\frac{m}{2}} + O\( e^{2\sqrt{AKn}}n^{-\frac{1}{2}(B+2)}\),
\end{align*}
where 
\begin{align*}
	a_m = \frac{A^{\frac{m}{2}}\(\frac{1}{2}+B+1\)_m\(\frac{1}{2}-(B+1)\)_m}{m!}.
\end{align*} 
We identify the above asymptotic with a Taylor series as in the proof of \cite[Theorem 5]{griffin2019jensen}. Thus, 
\begin{align*}
	C(n)\sim \lambda n^{-\frac{2B+3}{4}}e^{2\sqrt{AKn}}\exp\(c_0+c_1n^{-\frac{1}{2}}+c_2 n^{-1}+c_3 n^{-\frac{3}{2}}\ldots \)
\end{align*}
for some choice of constants $\lambda$ and $c_i$. Therefore, we obtain that  \[ \log\(\frac{C(n+k)}{C(n)}\) \sim 2\sqrt{AK}\sum_{m=1}^\infty {{1/2}\choose{m}}\frac{k^m}{n^{m-\frac{1}{2}}} + \frac{-2B-3}{4}\sum_{m=1}^\infty \frac{(-1)^{m-1}k^m}{mn^m} + \sum_{m,l\geq 1} c_l{{-l/2}\choose{m}}\frac{k^m}{n^{m+\frac{l}{2}}}.\]Upon considering, 
\begin{align*}
A(n) =& (AKn^{-1})^{\frac{1}{2}}+\frac{(2B+3)}{4n}, \\ \delta(n) =& \sqrt{2}(AK)^{\frac{1}{4}}n^{-\frac{3}{4}} + O \( n^{-\frac{5}{4}} \), \\ g_j(n) =& \frac{(-2B-3)(-1)^{j-1}}{4jn^j}
\end{align*}
the desired conclusion is obtained by Theorem \ref{T:GORZJensen}.
\end{proof}

\section{Proof of Theorems \ref{T:AppH} and \ref{T:App}}\label{S:App}

Within this section, we apply our results to the two classes of infinite products given by $H_{r,t} (q)$ and $G_{a,p}(q)$. By Theorems \ref{T:Hyperbolic} and \ref{T:Turan}, it suffices to show that these products satisfy conditions (1) and (2) of Proposition \ref{P:Wright}. We begin by considering $H_{r,t}(q) = (q^t;q^t)_\infty^{-r}$.

\begin{proof}[Proof of Theorem \ref{T:AppH}]
	By the classical modular transformation behaviour of the $q$-Pochhammer symbol (which can be viewed in terms of the transformation behaviour of the Dedekind eta function - see e.g. 5.8.1 of \cite{CohenStromberg}), it is not difficult to show that
	\begin{align*}
		(e^{-tz} ; e^{-tz})^{-r}_\infty = \left( \frac{tz}{2\pi}\right)^{\frac{r}{2}} e^{\frac{\pi^2 r}{6tz}} + O(\vert z\vert),
		\end{align*}
	in accordance with the first condition of Proposition \ref{P:Wright}. Furthermore, using the same arguments as in \cite{bringmann2016dyson}, we may show that $H_{r,t}(q)$ satisfies the second condition of Proposition \ref{P:Wright}. To this end, observe that \[\Log(H_{r,t}(q)) = -r\sum_{n=0}^\infty \log\(1-q^{t(n+1)}\) = r\sum_{n,k=1}^\infty \frac{q^{kt(n+1)}}{k} = r\sum_{k=1}^\infty \frac{q^{kt}}{k}\sum_{n=1}^\infty (q^{kt})^n = r\sum_{k=1}^\infty \frac{q^{kt}}{k(1-q^{kt})}.\]So, for each integer $1\leq h\leq t$, we have that \begin{align*}
\left\lvert \Log(H_{r,t}(\zeta_t^hq)) \right \rvert & = r\left|\sum_{k=1}^\infty \frac{q^{tk}}{k(1-q^{tk})}\right| \\
& \leq r\(\left| \frac{q^t}{1-q^t}\right| - \frac{|q|^t}{1-|q|^t} +  \sum_{k=1}^\infty\frac{|q|^{kt}}{1-|q|^{kt}}\)\\
& \leq r\(\left| \frac{q^t}{1-q^t}\right| - \frac{|q|^t}{1-|q|^t} + \log\( P(|q|^t)\)\)
\end{align*}where $P(q) \coloneqq 1/(q; q)_\infty$ denotes the generating function of partitions. It is well-known (for instance, see \cite[Lemma 2.1]{bringmann2016dyson}) that \[ \log(P(|q|^t)) = \frac{\pi^2}{6tx}+\frac{1}{2}\log\( \frac{tx}{2\pi}\) + O(x)\]upon expressing $q= e^{-z}$ and $z=x+iy\in\C$. Thus, it suffices to show that there is some constant $C>0$ such that\[ \left| \frac{q^t}{1-q^t}\right| - \frac{|q|^t}{1-|q|^t} < -\frac{C}{x}\]on the bounded cone $Mx\leq y\leq \frac{2\pi}{t}-Mx$. For these purposes, we will require a few cases.

Suppose that either $Mx\leq y\leq \frac{\pi}{2t}$ or that $\frac{3\pi}{2t}\leq y\leq \frac{2\pi}{t}-Mx$. In which case, we have that $0\leq\cos(ty)\leq \cos(Mtx)$ and consequently, \[|1-q^t|^2 = 1 - 2e^{-tx}\cos(ty)+e^{-2tx} \geq 1-2e^{-tx}\cos(Mtx)+e^{-2tx}.\]Therefore, we have that \[ \frac{1}{|1-q^t|} = \frac{1}{xt\sqrt{M^2+1}}+O(1).\]Next, suppose that $\frac{\pi}{2t}\leq y\leq \frac{3\pi}{2t}$. Then we see that $\cos(ty)\leq0$ and thus,\[ \frac{1}{|1-q^t|}\leq \frac{1}{\sqrt{1+e^{-2tx}}}\leq \frac{1}{xt\sqrt{M^2+1}}.\]Next, notice that \[ \frac{1}{1-|q|^t} = \frac{1}{1-e^{-tx}} = \frac{1}{xt}+O(1),\]and so we may combine the two bounds to obtain that \[\frac{1}{|1-q^t|}-\frac{1}{1-|q|^t} = \frac{1}{xt}\( \frac{1}{\sqrt{M^2+1}}-1\) + O(1).\]Upon exponentiating, we conclude that \[ |H_{r,t}(q)|\leq \sqrt{\frac{xt}{2\pi}}\exp\(\frac{r}{xt}\( \frac{\pi^2}{6} + \frac{1}{\sqrt{M^2+1}}-1\)\)(1+O(x)).\]
\end{proof}

We devote the remainder of this section to showing that \[G_{a,p}(q) \coloneqq \frac{1}{(q^{a} ; q^{p})_\infty}\]satisfies the conditions of Proposition \ref{P:Wright} whenever $p$ is prime and $0\leq a<p$. To do so, we will apply the Euler-Maclaurin summation formula stated in Subsection \ref{SS:EM}.

\begin{proof}[Proof of Theorem \ref{T:App}]
We first observe that \[\Log(G_{a,p}(q)) = -\sum_{n=0}^\infty \Log\(1-q^{a+pn}\) = \sum_{n,k=1}^\infty \frac{(q^{a+pn})^k}{k} = \sum_{k=1}^\infty \frac{q^{ak}}{k}\sum_{n=1}^\infty (q^{pk})^n = \sum_{k=1}^\infty \frac{1}{k}\(\frac{q^{ak}}{1-q^{pk}} \).\]

To start, we check that $G_{a,p}(q)$ satisfies the first constraint of Proposition \ref{P:Wright} with a major arc centred at $q = \zeta_p^h$ for each integer $1\leq h\leq p$. To this end, observe that\[ \Log(G_{a,p}(\zeta_p^h q)) = \sum_{\alpha=1}^p \zeta_p^{ha\alpha}\sum_{\ell\geq0} \frac{q^{a(\alpha+p\ell)}}{(\alpha+p\ell)(1-q^{p(\alpha+p\ell)})}.\]
Defining \[f(x) = \frac{e^{-ax}}{x(1-e^{-px})},\]we may easily see that \[f(x)\sim \sum_{n=-2}^\infty d(n)x^n\]where $d(-2) = \frac{1}{p}$ and $d(-1) = \frac{1}{2} - \frac{a}{p}$. Expressing $q = e^{-z}$, we have that 
\begin{align*}
	\sum_{\ell\geq 0} \frac{q^{a(\alpha+p\ell)}}{(\alpha+p\ell)(1-q^{p(\alpha+p\ell)})} = z\sum_{k=0}^\infty f\(pz\(\ell+\frac{\alpha}{p}\)\)
\end{align*} 
where $1\leq\alpha\leq p$. By applying Lemma \ref{L:EM}, we obtain that \[z\sum_{k=0}^\infty f\(pz\(\ell+\frac{\alpha}{p}\)\)\sim \frac{1}{zp}\zeta\(2, \frac{\alpha}{p}\)+I_{f, a}^* - \left(\frac{1}{2} - \frac{a}{p}\right) \( \Log(pz) + \psi\(\frac{\alpha}{p}\)+\gamma\) - \sum_{n=0}^\infty \frac{d(n)B_{n+1}\( \frac{\alpha}{p}\)}{n+1}z^{n+1}\]and so, \begin{align*}
\Log(G_{a,p}(\zeta_p^h q)) & = \sum_{\alpha=1}^p \zeta_p^{ha\alpha}z\sum_{\ell\geq 0} f\(pz\(\ell+\frac{\alpha}{p}\)\)\\
& = \sum_{\alpha=1}^p \zeta_p^{ha\alpha}  \left[\frac{\zeta\(2, \frac{\alpha}{p}\)}{pz} + I_{f, a}^* - \( \frac{1}{2} - \frac{a}{p}\)\( \Log(pz)+\psi\( \frac{\alpha}{p}\) + \gamma\)\right]+ O(|z|)\\
& = \frac{1}{pz}\( \sum_{\alpha=1}^p \zeta_p^{ha\alpha}\zeta\(2, \frac{\alpha}{p} \)\) + \(\sum_{\alpha=1}^p \zeta_p^{ha\alpha}I_{f,a}^*\)\\
& \ \ \ \ \ \ \ \ \ \ \ \ \ \ \ \ \ \ \ \ \ \ \ \ - \(\frac{1}{2} - \frac{a}{p}\) \( \sum_{\alpha=1}^p \zeta_p^{ha\alpha}\( \Log(pz) +\psi\( \frac{\alpha}{p}\) + \gamma\)\) + O(|z|).
\end{align*}Observe that if $h=p$, then we obtain that \[ \Log(G_{a,p}(\zeta_p^hq))\sim  \frac{\zeta(2, 1)}{pz} + p I_{f,a}^* - \(\frac{1}{2} - \frac{a}{p}\)\Log(pz) + O(|z|)\]as $\psi(1) +\gamma=0$. By making a change of variables $u\mapsto ua$ and applying \cite[Lemma 2.3]{bringmann2022distributions}, we see that \begin{align*}
I_{f,a}^* & = \frac{1}{p}\int_0^\infty \(\frac{e^{-ua}}{u(1-e^{-up})} - \frac{1}{p}u^{-2} - \( \frac{1}{2}-\frac{a}{p}\)\frac{e^{-au}}{u}\)du\\
& = \frac{1}{p}\int_0^\infty \(\frac{e^{-u}}{u(1-e^{-\frac{p}{a}u})} - \frac{a}{p}u^{-2} - \(\frac{1}{2}- \frac{a}{p}\)\frac{e^{-u}}{u} \)du\\
& = \frac{1}{p}\(\log\( \Gamma\( \frac{a}{p}\)\) + \( \frac{1}{2}-\frac{a}{p}\)\log\( \frac{a}{p}\) - \frac{1}{2}\log(2\pi)\).
\end{align*}
 We thus see that 
\begin{align*}
	G_{a,p}(e^{-z}) \sim \frac{\Gamma\left(\frac{a}{p}\right) \left(\frac{a}{p}\right)^{\frac{1}{2}-\frac{a}{p}}}{\sqrt{2\pi}} (pz)^{\frac{1}{2}-\frac{a}{p}} e^{\frac{\pi^2}{6pz}} +O(|z|)
\end{align*}
as $z \to 0^+$ on the major arc.

Next, we will assume that $h\neq p$. In which case, we have that $\sum_{\alpha=1}^p \zeta_p^{ha\alpha}=0$ as $p$ is prime. So, \[ \Log(G_{a,p}(\zeta_p^h q)) \sim \frac{1}{pz}\( \sum_{\alpha=1}^p \zeta_p^{ha\alpha}\zeta\( 2, \frac{\alpha}{p}\)\) - \( \frac{1}{2}-\frac{a}{p}\)\( \sum_{\alpha=1}^p \zeta_p^{ha\alpha}\psi\(\frac{\alpha}{p}\)\) + O(|z|).\]From the identity \cite[p. 39]{campbell1966integrales}, we see that \[ \sum_{\alpha=1}^p \zeta_p^{ha\alpha}\psi\( \frac{\alpha}{p}\) = p\Log(1-\zeta_p^{ha}).\]Additionally, we have that \[ \sum_{\alpha=1}^p\zeta_p^{ha\alpha}\zeta\( 2, \frac{\alpha}{p}\) = p^2\zeta_p^{ha} \Phi(\zeta_p^{ha}, 2, 1)\]where $\Phi(z,s,a):= \sum_{n=0}^\infty\frac{z^n}{(n+a)^s}$ is \emph{Lerch's transcendent}. We conclude that\[\Log(G_{a,p}(\zeta_p^hq)) \sim \frac{\zeta_p^{ha}\Phi(\zeta_p^{ha}, 2, 1)}{z} - \(\frac{1}{2}-\frac{a}{p}\)\Log(1-\zeta_p^{ha}) + O(|z|).\]Finally, observe that

 \[ \re{\zeta_p^{ha}\Phi(\zeta_p^{ha}, 2, 1)} = \frac{\pi^2}{6} - \frac{\pi^2ha}{p} + \frac{\pi^2(ha)^2}{p^2} < \frac{\pi^2}{6}\] by \cite[25.12.8]{NIST:DLMF}. Arguing as in the proof of Theorem \ref{T:AppH}, one may ascertain that \[ |G_{a,p}(q)|\leq \sqrt{\frac{xp}{2\pi}}\exp\(\frac{1}{xp}\( \frac{\pi^2}{6} + \frac{1}{\sqrt{M^2+1}}-1\)\)(1+O(x))\]over the minor arcs.
 
 We then apply Proposition \ref{P:Wright} with $K=1$, $A = \frac{\pi^2}{6p}$, $B=\frac{1}{2} -\frac{a}{p}$ and $\alpha_0 =   \frac{\Gamma\left(\frac{a}{p}\right) \left(\frac{a}{p}\right)^{\frac{1}{2}-\frac{a}{p}}p^{\frac{1}{2} -\frac{a}{p}}}{\sqrt{2\pi} }$ to conclude the asymptotic form of the coefficients, and thus Theorems \ref{T:Hyperbolic} and \ref{T:Turan} for the remaining results.
\end{proof}

\end{document}